\documentclass[12pt,oneside,english]{amsart}
\usepackage[T1]{fontenc}
\usepackage[latin1]{inputenc}
\usepackage{graphicx}
\usepackage{graphics}
\usepackage{geometry}
\usepackage{amssymb}
\usepackage{amsmath, amsthm}
\usepackage{verbatim}
\usepackage{hyperref}

\makeatletter

 \theoremstyle{plain}
 \newtheorem{thm}{Theorem}[]
 \numberwithin{equation}{section} 
 \numberwithin{figure}{section} 
 \theoremstyle{plain}
 \theoremstyle{definition}
 
 \theoremstyle{plain}
 \theoremstyle{plain}
 \newtheorem{lem}[thm]{Lemma} 
 \theoremstyle{remark}
 \newtheorem{rem}[thm]{Remark}
 \theoremstyle{plain}

\newcommand {\R}{\mathbb{R}}

\usepackage{babel}
\addto\extrasspanish{\bbl@deactivate{~}}
\makeatother
\title{Radial Symmetry of solutions to diffusion equations with discontinuous nonlinearities}
\author{Joaquim Serra}
\thanks{The author was supported by grants MTM 2008-06349-C03-01 (Spain) and 2009SGR-345 (Catalunya)}

\address{Universitat Politècnica de Catalunya\\ Departament de Matemàtica Aplicada I \\
Diagonal 647, 08028 Barcelona, Spain}
\email{joaquim.serra@upc.edu}

\begin{document}
\begin{abstract}
We prove a radial symmetry result for bounded nonnegative solutions to the $p$-Laplacian semilinear equation $-\Delta_p u=f(u)$ posed in a ball of $\R^n$ and involving discontinuous nonlinearities $f$. When $p=2$ we obtain a new result which holds in every dimension $n$ for certain positive discontinuous $f$. When $p\ge n$ we prove radial symmetry for every locally bounded nonnegative $f$. Our approach is an extension of a method of P. L. Lions for the case $p=n=2$. It leads to radial symmetry combining the isoperimetric inequality and the Pohozaev identity.
\end{abstract}

\maketitle
\parskip = 0pt
\parindent = 12pt

\section{Introduction}
We consider positive solutions of
\begin{equation}\label{equaciolap0}
\left\{\begin{array}{ll}
-\Delta u = f(u) \quad& \mbox{in } \Omega \subset \R^n\,, \\
u>0 & \mbox{in } \Omega\,, \\
u=0 & \mbox{on } \partial \Omega\,,
\end{array}
\right.
\end{equation}
where $\Omega$ is a ball. A classical theorem of Gidas-Ni-Nirenberg \cite{GiNiNir} states that if $f=f_1+f_2$ with $f_1$ Lipschitz and $f_2$ nondecreasing, then a solution $u\in C^2(\overline{\Omega})$ to \eqref{equaciolap0} has radial symmetry. Since $f_2$ might be any nondecreasing function, this result allows $f$ to be discontinuous, but only with increasing jumps. Besides this, the only other general result for $f$ discontinuous is, to our knowledge, the one of P.~L.~Lions \cite{Li}, that states radial symmetry of solutions for every locally bounded $f\ge 0$ in dimension $n=2$. 

In this paper we establish radial symmetry of solutions to \eqref{equaciolap0} in every dimension $n\ge 3$ under the assumption
\[\phi\le f\le \frac{2n}{n-2}\,\phi\]
for some nonincreasing function $\phi\ge 0$. In addition, we also obtain results for the $p$-Laplacian equation 
\begin{equation}\label{equaciou}
\left\{\begin{array}{l}
-\Delta_p u := -\nabla\cdot(|\nabla u|^{p-2} \nabla u) = f(u)\quad \text{in }\Omega\,,
\\ u\ge 0 \quad \text{in } \Omega\,,
\\ u=0 \quad \text{on } \partial \Omega\,,
\end{array}\right.
\end{equation}
where $\Omega\subset\R^n$ is a ball. For instance, under the assumption $p\ge n$, we establish radial symmetry of bounded solutions to \eqref{equaciou} for every $f\ge 0$ locally bounded but possibly discontinuous.

The result to be proved in this paper is the following:
\begin{thm}\label{mainresult}
Let $\Omega$ be a ball in $\R^n$, $n\ge 2$, and let $1<p<\infty$. Assume that $f\in L^\infty_{\text{loc}}([0,+\infty))$ is nonnegative. Let $u\in C^1(\Omega)\cap C^0(\overline \Omega)$ be a solution of $\eqref{equaciou}$ in the weak sense. Assume that either
\begin{enumerate}
\item[(a)]  $p\ge n$, 
\end{enumerate}
or 
\begin{enumerate}
\item[(b)]  $p<n$ and, for some nonincreasing function $\phi\ge0$, we have $\phi\le f\le \frac{np}{n-p} \phi\,.$
\end{enumerate}  

Then, $u$ is a radially symmetric and nonincreasing function. Moreover, $\frac{\partial u}{\partial r}<0$ in $\{0<u<\max_{\overline \Omega} u\}$, that will be an annulus or a punctured ball.
\end{thm}

This result follows the approach introduced in 1981 by P. L. Lions within the paper  \cite{Li}, where the case $p=n=2$ of Theorem \ref{mainresult} is proved (also with the hypothesis $f\ge 0$). In the same direction, Kesavan and Pacella \cite{KePa} established the cases $p=n\ge 2$ of Theorem \ref{mainresult}. In Lions' method, the isoperimetric inequality and the Pohozaev identity are combined to conclude the symmetry of $u$.

For some nonlinearities $f$ which change sign, there exist positive solutions of \eqref{equaciou} in a ball which are not radially symmetric, even with $p=2$ and $f$ Hölder continuous (see \cite{Br2} for an example).


For $1<p<\infty$, assuming that $f$ is locally Lipschitz and positive, and that $u\in C^1(\overline \Omega)$ is a positive solution of \eqref{equaciou} in a ball, Damascelli and Pacella \cite{DaPa} ($1<p<2$) and Damascelli and Sciunzi \cite{DaSc} ($p>2$) succeeded in applying the moving planes method to prove the radial symmetry of $u$. 

Another symmetry result for \eqref{equaciolap0} with possibly non-Lipschitz $f$ is due to Dolbeault and Felmer \cite{DF}. They assume that $f$ is continuous and that, in a neighborhood of each point of its domain, $f$ is either decreasing, or is the sum of a Lipschitz and a nondecreasing functions. If, in addition, $f\ge 0$, solutions $u\in C^1(\Omega)\cap C^0(\overline \Omega)$ to \eqref{equaciolap0} in a ball are radially symmetric. A similar result for the $p$-Laplacian equation \eqref{equaciou} is found in \cite{DFM}. These results use a local version of the moving planes technique. 

Under the weaker assumption that $f\ge 0$ is only continuous, for $1<p<\infty$, Brock \cite{Br} proved that $C^1(\overline \Omega)$ positive solutions of \eqref{equaciou} are radially symmetric using the so called ``continuous Steiner symmetrization''.
 
The radial symmetry results in \cite{Br} (via continuous symmetrization) and in \cite{DF,DFM} (via local moving planes) follow from more general local symmetry results \cite{Br2,DF,DFM} which do not require $f\ge 0$. These describe the only way in which radial symmetry may be broken through the formation of ``plateaus'' and radially symmetric cores placed arbitrarily on the top of them. The notion of local symmetry, introduced by Brock in \cite{Br2}, is very related to rearrangements. Nevertheless, in \cite{DF,DFM}, local symmetry results are proved using a local version of the moving planes method. 

Our technique leads to symmetry only when $\Omega$ is a ball. Instead, the technique used in \cite{Br}, as well as the moving planes method used in \cite{GiNiNir,DaSc,DF,DFM} are still applicable when the domain is not a ball, but is symmetric about some hyperplane and convex in the normal direction to this hyperplane. See \cite{BN} for an improved version of the moving planes method that allowed to treat domains with corners.

A feature of the original moving planes method in \cite{GiNiNir} and \cite{DaSc} is that, in addition to the radial symmetry, leads to $\frac{\partial u}{\partial r}<0$, for $r=|x|\in(0,R)$, $R$ being the radius of the ball $\Omega$. However, with discontinuous $f$ we cannot expect so much, even with $p=2$. A simple counterexample is constructed as follows: let $v$ be the solution of 
\[\left\{\begin{array}{l}
-\Delta_p v =1\quad \text{in }A=\{1/2<r<1\},
\\ v=0 \quad \text{on } \partial A\,.
\end{array}\right.\]
Then, $v$ is radial and positive, and thus it attains its maximum on a sphere $\{r=\rho_0\}$, for some $\rho_0\in(1/2,1)$. We readily check that $u=v\chi_{\{r>\rho_0\}} + (\max_{\overline A} v)\chi_{\{r\le\rho_0\}}$ is a solution of $\eqref{equaciou}$ for $\Omega=\{r<1\}$ and $f= \chi_{[0,\max_{\overline A} v)}\ge 0$,  and $u$ is constant on the ball $\{r\le\rho_0\}$.

Related to this, Theorem \ref{mainresult} states that $u$ is radial with $\frac{\partial u}{\partial r}<0$ in the annulus or punctured ball $\{ 0< u <\max_{\overline \Omega} u\}$ (see Lemma \ref{lema3}). Nevertheless, $u$ might attain its maximum in a concentric ball of positive radius $\{u=\max_{\overline \Omega} u\}$, as occurs in the preceding example. 

The following three distribution-type functions will play a central role in our proof:
\begin{equation} \label{IJK}
I(t)= \int_{\{u > t\}} f(u)\,d\mathcal H^n\,,\quad J(t)  = \mathcal H^n\bigl(\{u > t\}\bigr)\,,\quad  K = I^\alpha J^\beta\,. 
\end{equation}
These functions are defined for $t\in(-\infty,M)$, where $M=\max_{\Omega} u$. The parameters $\alpha$, $\beta$ in \eqref{IJK}, that are appropriately chosen depending on $p$ and $n$, are given by
\begin{equation}\label{alphaibeta1}
\alpha= p'= \frac{p}{p-1}\,,\quad \beta=  \frac{p-n}{n(p-1)}\,.
\end{equation}

Lions \cite{Li} in the case $p=n=2$ and Kesavan-Pacella \cite{KePa} in the cases $p=n\ge 2$ used the distribution type function $K=I^\alpha$ (note that our $\beta=0$ in these cases). By considering the function $K=I^\alpha J^\beta$ we are able to treat the cases $p\neq n$.

Observe that for any $t<0$ the value of $K(t)$ is equal to the constant 
\begin{equation}\label{equacioKzero}
K(0^-) = \lim_{t\to 0^{-}} K(t) = \biggl(\int_{\Omega} f(u)\,d\mathcal H^n\biggr)^{\negmedspace \alpha}(\mathcal H^n(\Omega))^\beta\,.
\end{equation} 
\begin{rem}\label{remark}
As we shall see, it is essential for our argument to work that the function $K$ in \eqref{IJK} be nonincreasing. This is trivially the case when $\alpha, \beta$ given by \eqref{alphaibeta1} are nonnegative, and thus this occurs when $p\ge n$. 

However, it may happen that, even with $\beta$ being negative, $K$ could be nonincreasing. This situation occurs under assumption (b) of Theorem \ref{mainresult}, i.e., $1<p<n$ and $\phi\le f \le\frac{pn}{n-p} \phi$ for some nonincreasing function $\phi\ge 0$. Indeed, in Lemma \ref{lema1} (iii) we will prove that, in this case, $K$ is absolutely continuous. Thus, to  verify that $K$ is nonincreasing we need to prove that $-K'\ge 0$ a.e.

Now, using statement \eqref{igualtat-K'} of the lemma, we obtain 
\[-K' =  \bigl\{\alpha I^{\alpha-1} J^\beta f + \beta I^\alpha J^{\beta-1}\bigl\} (-J')
= \bigl\{\alpha f + \beta I/J\bigl\} I^{\alpha-1} J^{\beta}(-J')\quad \mbox{a.e.}\]
From this we see that $-K'(t)$ has the same sign as $\bigl\{\alpha f(t) + \beta I(t)/J(t)\bigl\}$, since $I,\,J,\,-J'$ are nonnegative by definition. Thus, since $\beta<0$, we need $I(t)/J(t) + (\alpha/\beta) f(t) \le 0$ a.e. Observing that $I(t)/J(t)$ is the mean of $f(u)$ over the superlevel set $\{u>t\}$, we easily conclude that a sufficient condition for $I/J + (\alpha/\beta) f \le 0$ is that $f(s) \le -(\alpha/\beta)f(t)$, whenever $s>t$. And this is satisfied if $\phi\le f\le -(\alpha/\beta)\phi$ for some nonincreasing $\phi\ge 0$. Replacing $\alpha,\beta$ by their values in \eqref{alphaibeta1} we obtain the condition (b) in Theorem \ref{mainresult} since $-\alpha/\beta=pn/(n-p)$.  
\end{rem} 

\begin{rem}
Although the statement of Theorem \ref{mainresult} concerns solutions
of \eqref{equaciou} that are $C^1(\Omega)$, the arguments we shall use in its proof are often performed, in a standard way, with functions that are only of bounded variation. Nevertheless, from regularity results for degenerate elliptic equations of the type \eqref{equaciou}, we have that every bounded solution to \eqref{equaciou} is $C^{1,\alpha}(\Omega)$ for some $\alpha>0$. See, for instance, Lieberman \cite{Lb}. Thus, there is no loss of generality in assuming, in Theorem \ref{mainresult}, that $u \in C^1(\Omega)$ and this will turn some parts of its proof less technical. 
\end{rem}

\section{Preliminaries and proof of Theorem \ref{mainresult}}
All the technical details that will be needed in the proof of Theorem \ref{mainresult} are contained in the following three lemmas. The two first of them would be immediate if we assumed that $u$ and its level sets were regular enough. The third one leads to the radial symmetry of $u$ and the property $\frac{\partial u}{\partial r}<0$ in the annulus $\{ 0< u <\max_{\Omega} u\}$. The arguments used in their proofs are rather standard: for example, a finer version of inequality \eqref{desJ't} can be found in \cite{BrZi}. Nevertheless, we include them here to give a more self-contained treatment. 

\begin{lem}\label{lema1}
Let $\alpha,\beta$ be arbitrary real numbers and $\Omega\subset\R^n$ a bounded smooth domain. Assume that $u\in C^1(\Omega)\cap C^0(\overline \Omega)$ is nonnegative and $u|_{\partial\Omega}\equiv 0$. Let $f\in L^\infty_{\text{loc}}([0,+\infty))$ and let $I, J, K$ are defined by \eqref{IJK}. Let $M=\max_{\overline \Omega}u$. Then:

(i) The functions $I$, $J$ and $K$ are a.e. differentiable and  
\begin{equation} \label{igualtat-K'}
-K'(t) =  \bigl\{\alpha I(t)^{\alpha-1} J(t)^\beta f(t) + \beta I(t)^\alpha J(t)^{\beta-1}\bigl\} (-J'(t))\,\quad\text{for a.e. }t\,. 
\end{equation}

(ii) For a.e. $t\in(0,M)$, we have $\mathcal H^{n-1}\bigl(u^{-1}(t)\cap\{|\nabla u|=0\}\bigr)=0$ and
\begin{equation}\label{desJ't}
-J'(t) \ge \int_{u^{-1}(t)} \frac{1}{|\nabla u|}\, d\mathcal H^{n-1}\,.
\end{equation}

(iii) Assume furthermore that hypothesis (b) in Theorem \ref{mainresult} holds and that $u$ is a weak solution of \eqref{equaciou}. Then, $I$, $J$ and $K$ are absolutely continuous functions for $t<M$.  

\begin{proof}
(i) The functions $I$ and $J$ are nonincreasing by definition and hence differentiable almost everywhere. Furthermore, they define nonpositive Lebesgue-Stieltjes measures $dI$ and $dJ$ on $(0,M)$. By definition of Lebesgue integral, using approximation by step functions, we find that 
\[I(t)= \int_{\{u\ge t\}} f(u)\,d\mathcal H^n = -\int_t^{M^+} f(t) dJ(t)\]
and hence $dI= f dJ$. From this, it follows that $I'(t) = f(t)\,J'(t)$ for $dJ$-a.e.$\,t$ in $(0,M)$.  But since $|\nabla u|$ is bounded in $\{u\ge t\}$, $J$ is strictly decreasing. This leads to $\mathcal L\ll dJ$, where $\mathcal L$ is the Lebesgue measure in $(0,M)$. Therefore we have  $I'(t) = f(t)\,J'(t)$ for a.e. $t$ (in all this paper, unless otherwise indicated, a.e. is with respect to the Lebesgue measure). As a consequence, \eqref{igualtat-K'} holds. 

(ii) Start defining 
\[ 
J_0(t) = \mathcal H^n\bigl(\{u > t, |\nabla u|>0\}\bigr)\,. 
\]
Let $\epsilon>0$ and $T \in(0,M)$. Let $u_T=\max(u,T)$. We extend $u_T$ outside $\Omega$ by the constant $T$, to obtain a Lipschitz function defined in all $\R^n$. Applying to $u_T$ the coarea formula for Lipschitz functions (see, for example, Theorem 2 in sec.  3.4.3 of  \cite{EvGa}), we can compute
\begin{equation}\label{pork}
\begin{split}
\mathcal H^n(\{u > T, |\nabla u|>\epsilon\})&= \int_{\R^n} |\nabla u_T| \frac{\chi_{\{u > T, |\nabla u|>\epsilon\}}}{|\nabla u|} \, d\mathcal H^n \\
&= \int_T^M \int_{u^{-1}(t)}\frac{\chi_{\{u > T, |\nabla u|>\epsilon\}}}{|\nabla u|}\, d\mathcal H^{n-1}\, dt\\
&= \int_T^M \int_{u^{-1}(t)\cap\{|\nabla u|>\epsilon\}} \frac{1}{|\nabla u|}\, d\mathcal H^{n-1}\, dt \,.
\end{split}
\end{equation}
For any given $\epsilon>0$, $|\nabla u|^{-1} \,\chi_{\{u > T, |\nabla u|>\epsilon\}}$ is $\mathcal H^n$-summable. Now, by monotone convergence, letting $\epsilon\to 0$ in \eqref{pork} we find that \eqref{pork} also holds for $\epsilon=0$ (and arbitrary $T$). We deduce that 
$J_0(t)$ is an absolutely continuous function and that
\begin{equation}\label{equacioJ'0}
-J_0'(t) = \int_{u^{-1}(t)\cap\{|\nabla u|>0\}} \frac{1}{|\nabla u|}\, d\mathcal H^{n-1}\,, \quad \text{for a.e. } t\in(0,M)\,. 
\end{equation}  

Applying one more time the coarea formula to $u_T$ we obtain
\[
0= \int_{\R^n} |\nabla u_T| \chi_{\{ u>T,|\nabla u|=0\}} = \int_T^M \mathcal H^{n-1}\bigl(u^{-1}(t)\cap\{|\nabla u|=0\}\bigr)\, dt \,.
\]
We conclude that, for a.e. $t$, the set $u^{-1}(t)\cap\{|\nabla u|=0\}$ has zero $\mathcal H^{n-1}$-measure. Having this into account we may change \eqref{equacioJ'0} for the apparently finer
\begin{equation}\label{equacioJ'0stgr}
-J_0'(t) = \int_{u^{-1}(t)} \frac{1}{|\nabla u|}\, d\mathcal H^{n-1}\,, \quad \text{for a.e. } t\in(0,M)\,. 
\end{equation}

Next, observe that for a.e. $t\in(0,M)$ (where both $J'(t)$ and $J'_0(t)$ exist) we have the inequality
\[
-J'(t) = \lim_{s\to t^+} \frac {\mathcal H^n\bigl(\{s\ge u> t\}\bigr)}{s-t}
\ge \lim_{s\to t^+} \frac {\mathcal H^n\bigl(\{s\ge u> t, |\nabla u|>0\}\bigr)}{s-t}
= -J_0'(t) \,.
\]
Combining this with \eqref{equacioJ'0stgr} we get finally \eqref{desJ't}. It is easily verified that equality holds in \eqref{desJ't} for a.e. $t$ if the set $\{|\nabla u|=0\}$ has zero $\mathcal H^n$-measure. 

(iii) If $p<n$ and $\phi\le f\le \frac{np}{n-p} \phi$ for some nonincreasing $\phi\ge0$, then a solution of $-\Delta_p u = f(u)$ will be $p$-harmonic in $\{u\ge t_0\}$, where $t_0\in[0,+\infty]$ satisfies that $\phi(t)>0$ for $t<t_0$ and $\phi(t)\equiv 0$ for $t>t_0$. Hence, if $t_0<+\infty$, we will  have that $u\equiv t_0=M=\max_{\overline{\Omega}}u$ in $\{u\ge t_0 \}$. Therefore, for every $t<  M=\max_{\overline{\Omega}}u$ we have that $-\Delta_p u = f(u) \ge \phi(u)\ge\phi(t) >0$ in $\{0\le u< t\}$. But since $f(u)\in L^\infty(\Omega)$, we can apply a result of  H. Lou, Theorem 1.1 in \cite{Lo} and find that $f(u)$ vanishes a.e. in the set $\{|\nabla u|=0\}\cap \{0\le u< t\}$. Since $f(u) \ge \phi(t)>0$ in $\{0\le u< t\}$, this is only possible if the singular set $\{|\nabla u|=0, u< t\}$ has zero measure. Therefore, we have $J(t)=J_0(t)+ \mathcal H^{n}(\{u=M\})$ for every $t<M$ and thus $J$ is an absolutely continuous function (since we have shown that $J_0$ is absolutely continuous), at least for $t<M$. From this, it is immediate to see that also $I$ and $K$ are absolutely continuous for $t<M$.
\end{proof}
\end{lem} 
 
\begin{lem}\label{lema2}
Under the assumptions of Theorem \ref{mainresult}, let $M=\max_{\overline{\Omega}} u$. We have the following:

(i)  It holds the Gauss-Green type identity
\begin{equation}\label{teodivergencia}
I(t) = \int_{u^{-1}(t)} |\nabla u|^{p-1} d\mathcal H^{n-1}, \quad \text{for a.e. } t\in(0,M)\,.
\end{equation}

(ii) It holds the isoperimetric inequality
\begin{equation}\label{isoperimetricinequality}
\mathcal H^{n-1}(u^{-1}(t)) \ge c_n \bigl(\mathcal H^{n}(\{u>t\})\bigr)^{\frac{n-1}{n}} = c_n J(t)^{\frac{n-1}{n}} \,, \quad \text{for a.e. } t\in(0,M)\,,
\end{equation}
where $c_n$ is the optimal isoperimetric constant in $\R^n$, $c_n=\mathcal H^{n-1}(\partial B)
\bigl( \mathcal H^{n}(B)\bigr)^{\frac{1-n}{n}}$ with $B$ being a ball  in $\R^n$.
\begin{proof}
(i) Since the function $u$ is of bounded variation locally in $\Omega$, we know from the coarea theorem for BV functions (see Theorem 1, sec. 5.5 of \cite{EvGa}) that the sets $\{u>t\}$ have finite perimeter for a.e. $t$.  For the measure theoretic boundary $\partial_*\{u>t\}$ (see section 5.8 of \cite{EvGa}), we readily check that $\{u=t, |\nabla u|>0\} \subset \partial_*\{u>t\} \subset u^{-1}(t)$. But recall from  Lemma \ref{lema1} (ii) that $\mathcal H^{n-1}\bigl(u^{-1}(t)\cap\{|\nabla u|=0\}\bigr)=0$ for a.e. $t$. We conclude that $\bigl(\int_{u^{-1}(t)} |\nabla u|^{p-1}\, d\mathcal H^{n-1}\bigr)$ and $\bigl(\int_{\partial_*\{u>t\}} |\nabla u|^{p-1}\, d\mathcal H^{n-1}\bigr)$ are equal for a.e. $t$.

On the other hand, the vector field $-\nabla u$ is perpendicular to the regular surface $u^{-1}(t)\cap\{|\nabla u|>0\}$. We have just seen that this regular surface fills almost all $\partial_*\{u>t\}$, in the sense of $\mathcal H^{n-1}$-measure. As a conclusion, if $\nu$ is the measure theoretical normal vector for $\{u>t\}$ then $-\nabla u\cdot\nu = |\nabla u|$ $\mathcal H^{n-1}$-a.e. on $\partial_*\{u>t\}$.
 
Since $u$ solves \eqref{equaciou}, by the generalized Gauss-Green theorem (Theorem 1, sec. 5.8. of \cite{EvGa}), we have 
\begin{equation}\label{111}
I(t) = \int_{\{u>t\}}f(u)d\mathcal H^n = \int_{\partial_*\{u>t\}} |\nabla u|^{p-1} d\mathcal H^{n-1}= \int_{u^{-1}(t)} |\nabla u|^{p-1} d\mathcal H^{n-1}\,,
\end{equation}
for a.e. $t\in(0,M)$. Although that the precise version of Gauss-Green theorem we cite applies to a $C^1_c(\Omega)$ vector field, and we only have $|\nabla u|^{p-2}\nabla u\in C^0(\Omega)$, this can easily be handled as follows. Given $t$, we approximate uniformly in $\{u\ge t\}$ the continuous vector field $|\nabla u|^{p-2}\nabla u$ by a sequence of $C^1_c(\Omega)$ vector fields $(\phi_n)$ to which we can apply the theorem. Doing so $\nabla\cdot\phi_n$ converges weakly to $f(u)$, and this is enough for our purposes. Indeed, for each $\phi_n$ we have
\[\int_{\{u> t\}} \nabla\cdot \phi_n  = \int_{\partial_*\{u> t\}} \phi_n \cdot \nu \,d\mathcal H^{n-1}\,,\quad \text{for a.e. }t\in(0,M)\,.\]
Now, letting $n\to \infty$ we obtain \eqref{111}, and hence \eqref{teodivergencia}. 

(ii) We have seen that $\{u>t\}$ is a bounded set  of finite perimeter for a.e. $t\in(0,M)$. Thus the isoperimetric inequality \eqref{isoperimetricinequality} with the best constant follows immediately from Theorem 2 and the Remark that follows it in Section 5.6.2 of \cite{EvGa}.
\end{proof}
\end{lem}

Radial symmetry will follow from next lemma after having proved that hypothesis (1) and (2) on it hold. For a detailed discussion on a very similar question see the article of Brothers and Ziemer \cite{BrZi}. Here, we present an ad hoc argument inspired by this article.

\begin{lem}\label{lema3}
Assume that $f\in L^\infty_{\text{loc}}([0,+\infty))$ is nonnegative. Let $\Omega$ be a ball in $\R^n$ and $u\in C^1(\Omega)\cap C^0(\overline{\Omega})$ be a solution of \eqref{equaciou} in the weak sense. Let $M=\max_{\overline{\Omega}} u$. Suppose that for a.e. $t\in(0,M)$
\begin{enumerate}
\item[(1)] $\{u>t\}$ is a ball (centered at some point in $\R^n$, possibly depending on $t$)
\end{enumerate}
and
\begin{enumerate}
\item[(2)] $|\nabla u|$ is constant on $\partial\{u>t\}$.
\end{enumerate}

Then, $u$ is radially symmetric. In addition, $\frac{\partial u}{\partial r} <0$ in the annulus or punctured ball $\{0<u<\max_{\overline{\Omega}}u\}$, but $u$ could achieve its maximum in a ball of positive radius.     

\begin{proof} 
Denote by $\Theta$ the set of $t\in(0,M)$ for which $\{u>t\}$ is a ball. As $\Theta$ is a dense subset (its complementary has zero measure), for any $t\in(0,M)$ we have $\{u>t\}= \bigcup_{s>t,\,s\in\Theta} \{u>s\}$. Thus, every superlevel set $\{u>t\}$ is a increasing union of balls with bounded diameter, hence it is also a ball. Therefore $\Theta=(0,M)$ and we have 
\[\{u>t\}= B(x(t); \rho(t))\]
for some $x(t)$, $\rho(t)$ defined for every $t\in(0,M)$.  

From the continuity of $\nabla u$ and hypothesis $(2)$ we deduce that $|\nabla u|$ is constant on $\partial B(x(t); \rho(t))$ for every $t\in(0,M)$. Besides, as $u$ is a solution of \eqref{equaciou}, the Gauss-Green theorem leads to
\[\mathcal H^{n-1}\bigl(\partial B(x(t); \rho(t))\bigr) |\nabla u|^{p-1}\bigl(\partial B(x(t); \rho(t))\bigr) = 
 \int_{B(x(t); \rho(t))} f(u)\,\,d\mathcal H^{n}\,.\] 
But $f(u)\ge 0$ and, by the maximum principle, it is impossible that $f(u)\equiv 0$ on some $\{u>t\}$. We conclude that $\nabla u$ does not vanish in the open set $\{0<u<M\}$. 

Having now that $u$ is a $C^1$ function whose gradient never vanishes in the open set $\{u<M\}$, it is easily shown that $J(t)=\mathcal H^n\bigl(\{u>t\}\bigr)$ is locally Lipschitz in $(0,M)$. Therefore, also $\rho(t) = \bigl(J(t)/\omega_n)^{1/n}$ is locally Lipschitz ($\omega_n= \mathcal H^n(B_1)$ is the volume of a unit ball in $\R^n$). Moreover, since $B(x(t);\rho(t))=\{u>t\}\supset\{u>s\} =B(x(s);\rho(s))$ for $t<s$, we deduce that
\begin{equation}\label{lultima}
|x(t)-x(s)|\le \rho(t)-\rho(s) \quad \mbox{for }t< s\,.
\end{equation}
Thus, $x=x(t)$ is also locally Lipschitz.    

Now suppose that $u$ were not radially symmetric. Then $x$ would not be identically constant in $(0,M)$ and hence we could find some $t_0\in(0,M)$ such that the velocity vector $y =\frac{d}{dt}x (t_0)$ would exist and be nonzero. But in such case, setting $z=y/|y|$, $P(t)=x(t)+\rho(t)z$ and $Q(t)= x(t)-\rho(t)z$, by hypothesis, we would have 
\[u(P(t))\equiv u(Q(t))\equiv t \quad \mbox{for all }t\,,\]
and $\nabla u(P(t_0))\cdot z = -|\nabla u(P(t_0))|$ while $\nabla u(Q(t_0))\cdot z = |\nabla u(Q(t_0))|$. This would lead to
\[1= \frac{d}{dt}\bigr|_{t_0} u(P(t)) = \nabla u(P(t_0)) \cdot (|y|+\rho'(t_0))z
   =  -|\nabla u(P(t_0))|(|y|+\rho'(t_0))\]
and 
\[1= \frac{d}{dt}\bigr|_{t_0} u(Q(t)) = \nabla u(Q(t_0)) \cdot (|y|-\rho'(t_0))z
   =  |\nabla u(Q(t_0))|(|y|-\rho'(t_0))\,.\]
But we must have $|\nabla u(P(t_0))|= |\nabla u(Q(t_0))|$ since both $P(t_0)$ and $Q(t_0)$ belong to $\partial B(x(t_0);\rho(t_0))=\partial\{u>t_0\}$. Then, it would follow that $|y|=0$, which is a contradiction.

As a consequence, $u$ is to be is radially symmetric. We already justified that $|\nabla u|$ does not vanish in $\{0<u<M\}$, hence $\frac{\partial u}{\partial r} <0$ in this open ring. However we may not discard the possibility of $u$ being constant on a closed non-degenerate ball $\{u=M\}$, as happens in the example given in Section 1.
\end{proof}
\end{lem}

Finally we present the proof of the result in this paper. 


\begin{proof}[Proof of Theorem \ref{mainresult}]
We first note that, under the assumptions of the theorem, $K(t)$ is nonincreasing for $t\in(0,M)$, where $M=\max_{\overline{\Omega}}u$, $K$ is given by \eqref{IJK} and $\alpha, \beta$ are given by \eqref{alphaibeta1}. Indeed, under hypothesis (a) of the theorem it is obvious because $\alpha,\beta\ge0$. On the other hand, under hypothesis (b) Lemma \ref{lema1} (iii) applies and hence $K$ is absolutely continuous. But, as shown in Remark \ref{remark}, $-K'\ge 0$ a.e. in this case. Therefore, $K$ is nonincreasing again.  

Since $K(t)$ is nonnegative and nonincreasing, we have (even if $K$ could have jumps)
\begin{equation}\label{desigualtatK}
K(0^-)\ge K(0^+)-K(M^-) \ge \int_0^{M} -K'(t)\,dt\,.
\end{equation}  

Combining \eqref{desigualtatK} and \eqref{igualtat-K'} in Lemma \ref{lema1} (i), we are lead to
\[
K(0^-)\ge \int_0^M \bigl\{\alpha I(t)^{\alpha-1} J(t)^\beta f(t) + \beta I(t)^\alpha J(t)^{\beta-1}\bigl\} (-J'(t)) \,dt\,.
\]
The integrand on the right equals $-K'(t)$ and hence is nonnegative for a.e.$\,t$. Also $-J'(t)$ is nonnegative. Therefore so is the factor in brackets  and we can use inequality \eqref{desJ't} to obtain a further estimate:
\begin{equation}\label{desigualtatK2}
K(0^-)\ge \int_0^M \bigl\{\alpha I(t)^{\alpha-1} J(t)^\beta f(t) + \beta I(t)^\alpha J(t)^{\beta-1}\bigl\} \biggl(\int_{u^{-1}(t)} \frac{1}{|\nabla u|}\, d\mathcal H^{n-1}\biggr) \,dt\,.
\end{equation}
Equalities are obtained when $K$ is absolutely continuous.    

Next we derive the following isoperimetric-Hölder type inequality:
\begin{equation}\label{superdesigualtat}
I(t)^{\frac{1}{p-1}}\biggl(\int_{u^{-1}(t)} |\nabla u|^{-1}\, d\mathcal H^{n-1}\biggr) \ge c_n^{\frac{p}{p-1}} J(t)^{\frac{p(n-1)}{(p-1)n}}\,,\quad \text{for a.e. } t\in(0,M)
\end{equation} 
with $c_n$ as in \eqref{isoperimetricinequality}. To prove \eqref{superdesigualtat}, we use \eqref{teodivergencia} in Lemma \ref{lema2} to conclude that, for a.e.$\,t$,
\[\begin{split}
I(t)^{\frac{1}{p}}\biggl(\int_{u^{-1}(t)} |\nabla u|^{-1}\,& d\mathcal H^{n-1}\biggr)^{\frac{p-1}{p}} =\\
&= \biggl(\int_{u^{-1}(t)} |\nabla u|^{p-1}\,d\mathcal H^{n-1}\biggr)^{\frac{1}{p}}\biggl(\int_{u^{-1}(t)}|\nabla u|^{-1}\,d\mathcal H^{n-1}\biggr)^{\frac{p-1}{p}}\\
&\ge \mathcal H^{n-1}\bigl(u^{-1}(t)\bigr) \ge c_n\mathcal H^n\bigl(\{u> t\}\bigr)^\frac{n-1}{n}
= c_n J(t)^\frac{n-1}{n}\,, 
\end{split}
\]
where the first inequality is a consequence of Hölder's inequality, and the second one of the isoperimetric inequality \eqref{isoperimetricinequality}. We emphasize that both equalities hold simultaneously if and only if $\{u> t\}$ is a ball and $|\nabla u|$ is constant on $u^{-1}(t)$.

Returning to \eqref{desigualtatK2}, we deduce from \eqref{superdesigualtat} 
\begin{equation}\label{desigualtatK3}
\begin{split}
K(0^-)&\ge \int_0^M  \bigl\{\alpha I^{\alpha-1} J^\beta f + \beta I^\alpha J^{\beta-1}\bigl\} \biggl(\int_{u^{-1}(t)} \frac{1}{|\nabla u|}\, d\mathcal H^{n-1}\biggr) \,dt\\
&= \int_0^M \bigl\{\alpha I^{\alpha-1- \frac{1}{p-1}} J^\beta f + \beta I^{\alpha- \frac{1}{p-1}} J^{\beta-1}\bigr\} I^{\frac{1}{p-1}}\biggl(\int_{u^{-1}(t)} \frac{1}{|\nabla u|}\, d\mathcal H^{n-1}\biggr)\,dt\\
&\ge \int_0^M c_n^{\frac{p}{p-1}} \bigl\{\alpha I^{\alpha-1- \frac{1}{p-1}} J^\beta f + \beta I^{\alpha- \frac{1}{p-1}} J^{\beta-1}\bigl\}  J(t)^{\frac{p(n-1)}{(p-1)n}} \,dt 
\,.
\end{split}
\end{equation}
For the last inequality we are using (for second time) that the factor in brackets in \eqref{desigualtatK2} is nonnegative, since $-K'\ge 0$.       

Note that in order to obtain \eqref{desigualtatK3} we are integrating the isoperimetric-Hölder inequality \eqref{superdesigualtat} over almost all the levels. Accordingly, 
\begin{rem}\label{remrem}
A necessary condition for having equalities in  \eqref{desigualtatK3}  is that, for a.e.$\,t\in(0,M)$, $\{u>t\}$ is a ball and $|\nabla u|$ is constant on $u^{-1}(t)$.
\end{rem}  

Next, the values of $\alpha$ and $\beta$ in \eqref{alphaibeta1} are set to satisfy that $\alpha-1- \frac{1}{p-1}=0$ and $\beta-1+ \frac{p(n-1)}{(p-1)n}=0$. Then \eqref{desigualtatK3} becomes 
\[
\begin{split}
K(0^-) &\ge \int_0^M  c_n^{p'} \biggl( p'\,f(t) \,\mathcal H^n\bigl(\{u > t\}\bigr) \,+  \, \frac{p-n}{n(p-1)} \int_{\{u > t\}} f(u)\,d\mathcal H^n \biggr)  \,dt\\
&= c_n^{p'} \int_0^M \int_\Omega   \chi_{\{u > t\}} \biggl(p'\,f(t) \ +  \, \frac{p-n}{n(p-1)} f(u) \biggr) \,d\mathcal H^n\,dt\\
&= c_n^{p'} \biggl( p'\negmedspace\int_{\Omega} F(u)\,d\mathcal H^n  +  \frac{p-n}{n(p-1)} \int_{\Omega} uf(u) \,d\mathcal H^n\biggr)\\
&=c_n^{p'}\frac{p'}{n}\biggl( n\int_{\Omega} F(u)\,d\mathcal H^n  +  \frac{p-n}{p} \int_{\Omega} uf(u)\,d\mathcal H^n \biggr)
\end{split}
\]
for $F(s)= \int_0^{s} f(s')\, ds'$. Recalling \eqref{equacioKzero} we obtain finally the inequality 
\begin{equation}\label{desigualtati}
\frac{n}{p'c_n^{p'}}\mathcal H^n(\Omega)^{\frac{p-n}{n(p-1)}}\biggl(\int_\Omega f(u)\biggr)^{p'}  \ge  n\negmedspace \int_{\Omega} F(u)  +  \frac{p-n}{p} \negthinspace\int_{\Omega} uf(u) \,.
\end{equation}

Now we use for first time that $\Omega$ is a ball. As in \cite{KePa}, a combination of Pohozaev's identity
\[
n\negthinspace \int_{\Omega} F(u)  +  \frac{p-n}{p} \negthinspace\int_{\Omega} uf(u) = \frac{1}{p'} \int_{\negthinspace\partial\Omega} (x\cdot\nu) \biggl|\frac{\partial u}{\partial \nu}\biggr|^p
\] 
and Hölder inequality gives, when  $\Omega$ is a ball of radius $R$, the inequality
\begin{equation}\label{desigualtatii}
n\negthinspace \int_{\Omega} F(u)  +  \frac{p-n}{p} \negthinspace\int_{\Omega} uf(u) \ge \frac{1}{(n\omega_n R^{n-1})^{p'/p}}\frac{R}{p'} \biggl(\int_\Omega f(u)\biggr)^{p'},
\end{equation}
where $\omega_n = \mathcal H^{n}(B_1)$ is the volume of the unit ball in $\R^n$.

To conclude, a straightforward computation (there is no magic behind this: note that all the inequalities obtained throughout this proof are equalities when $u$ is radial) and recalling the value of $c_n$ given in Lemma \ref{lema2} (ii), we check that 
\[\begin{split}
\frac{n}{p'c_n^{p'}} \mathcal H^n(\Omega)^{\frac{p-n}{n(p-1)}} &= \frac{n \mathcal H^n(B_R)^{\frac{p'(n-1)}{n}}}{p' \mathcal H^{n-1}(\partial B_R)^{p'}} \mathcal H^n(B_R)^{\frac{p-n}{n(p-1)}}  
= \frac{n (\omega_n R^{n})^{\frac{p'(n-1)}{n}+\frac{p-n}{n(p-1)}}}{p' (n\omega_n R^{n-1})^{p'}} \\&=\frac{1}{(n \omega_n R^{n-1})^{p'/p}}\frac{R}{p'}\,.
\end{split} \]
This enlightens that \eqref{desigualtati} and \eqref{desigualtatii} are opposite inequalities. Therefore they must be, in fact, equalities. 

It follows, recalling Remark \ref{remrem} within this proof, that for a.e. $t\in(0,M)$, 
\begin{enumerate}
\item[(1)] the level $\{u>t\}$ is a ball
\end{enumerate}
and
\begin{enumerate}
\item[(2)] $|\nabla u|$ is constant on $\partial\{u>t\}$.
\end{enumerate}
But then from Lemma \ref{lema3} we conclude that $u$ is a nonincreasing function of the radius and with $\frac{\partial u}{\partial r}<0$ in $\{0<u<\max_{\overline \Omega} u\}$.  
\end{proof}

\section*{Acknowledgment}
The author thanks Xavier Cabré for interesting discussions on the subject of this paper.

\end{document}